\documentclass[letterpaper,10pt,conference]{ieeeconf}

\IEEEoverridecommandlockouts
\usepackage{amsmath,amssymb}
\usepackage{mathrsfs}
\usepackage{mathtools}

\usepackage{graphicx}
\usepackage{tikz}
\usetikzlibrary{arrows.meta,positioning,shapes}
\usetikzlibrary{decorations.pathreplacing}
\usepackage{booktabs}
\usepackage{multirow}
\usepackage{pgfplots}
\usepackage{pgfplotstable}
\pgfplotsset{compat=1.18}
\usepackage{algorithm}
\usepackage{algpseudocode}
\usepackage{comment}
\usepackage{hyperref}

\usepackage{cite}

\newtheorem{thm}{Theorem}[section]
\newtheorem{prop}{Proposition}[section]
\newtheorem{rem}{Remark}[section]

\newtheorem{lem}{Lemma}[section]
\newtheorem{defn}{Definition}

\newtheorem{assump}{Assumption}
\newcommand{\boxend}{\hfill \ensuremath{\Box}}

\allowdisplaybreaks

\title{
Submodular Welfare under Routing Coupling: A Hierarchical Decomposition with Perturbation Guarantees
}

\author{%
  Joan Vendrell Gallart$^{1}$ \quad, Nhat-Minh Tang-Nguyen$^{1}$ \quad, Alan Kuhnle$^{2}$ \quad and \quad Solmaz Kia$^{1}$%
  \thanks{This work was supported in part by NSF.}%
  \\[0.3em]
  $^{1}$Mechanical and Aerospace Engineering, University of California, Irvine \\
  $^{2}$Computer Science and Engineering Department of the Texas A\&M University, Texas
}

\begin{document}
\maketitle

\begin{abstract}
This paper studies joint submodular welfare maximization and routing over graphs, where agents select items under diminishing returns and transport them through a network with congestion-dependent costs. Although welfare maximization admits matroid-based approximations and routing reduces to shortest paths under modular costs, their coupling creates supermodular interactions that break separability.
We show that, for fixed routing, the objective remains submodular in the allocation variable, enabling a principled decomposition. Building on this property, we propose the Welfare-based Hierarchical Routing Algorithm (\texttt{WHIRL}), which alternates between tractable routing and allocation updates. Routing is initialized through its modular counterpart, while supermodular effects are modeled as bounded perturbations.
The method has finite convergence guarantees and approximation bounds that depend explicitly on the deviation from modular routing. Numerical results illustrate the impact of routing-induced coupling and show that \texttt{WHIRL} achieves a favorable tradeoff between solution quality and computational cost.
\end{abstract}

\section{Introduction}
\label{sec::intro}

The \emph{Submodular Welfare Problem}~\cite{nemhauser1978analysis,feige2006allocation} is a fundamental model in combinatorial optimization, where $m$ discrete items in $\mathcal{Q}=\{1,\cdots,m\}$ are distributed among $n$ agents, each endowed with a monotone submodular utility function $f_i : 2^{\mathcal{Q}} \to \mathbb{R}_{\geq 0}$, $i\in\mathcal{P}=\{1,\cdots,n\}$, that captures diminishing marginal returns. The objective is to find a non-overlapping partition $(\mathcal{S}_1,\ldots,\mathcal{S}_n)$, $\cup_{j\in\mathcal{P}}\mathcal{S}_j=\mathcal{Q}$, maximizing the aggregate welfare $\sum_{i\in\mathcal{P}} f_i(\mathcal{S}_i)$. A classical approach reformulates this problem over a partition matroid by lifting the ground set to $X = \mathcal{P} \times \mathcal{Q}$ and defining
\begin{equation*}
    \mathcal{I} = \big\{\mathcal{S} \subseteq X :
    |\mathcal{S} \cap (\mathcal{P} \times \{j\})| \leq 1, \forall\, j\in\mathcal{Q}\big\},
\end{equation*}
which enforces that each item in $\mathcal{Q}$ is assigned to at most one agent in $\mathcal{P}$; see Fig.~\ref{fig:welfare_problem}. This reduction casts welfare maximization as submodular maximization under a matroid constraint, for which a $\tfrac{1}{2}$-approximation via the Sequential Greedy Algorithm (\texttt{SGA}) is achievable~\cite{vondrak2008optimal}.

A natural and practically relevant extension arises when agents are embedded in a transportation network. In particular, consider providers $i \in \mathcal{P}$ located on a directed weighted graph $\mathcal{G} = (\mathcal{V}, \mathcal{E}, w)$, where $|\mathcal{P}| = n$. Each provider $i$ must not only select a set of items $\mathcal{S}_i$, but also transport them to a designated client node $c \in \mathcal{V}$ via a feasible directed path $\mathcal{R}_i \in \Omega_i$. Here, $\Omega_i \subseteq 2^{\mathcal{E}}$ denotes the set of all available routes for agent $i$, where each route $\mathcal{R}_i$ is defined as a subset of edges $e \in \mathcal{E}$. The welfare objective is then augmented by a routing cost $T(\mathscr{S}, \mathscr{R})$, leading to the joint problem:
\begin{equation*}
    \max_{\mathscr{S},\,\mathscr{R}} \;\sum_{i\in\mathcal{P}} f_i(\mathcal{S}_i) - \alpha\, T(\mathscr{S}, \mathscr{R}),
\end{equation*}
where $\mathscr{S} = (\mathcal{S}_1,\ldots,\mathcal{S}_n)$ denotes the item partition, $\mathscr{R} = (\mathcal{R}_1,\ldots,\mathcal{R}_n)$ denotes the route assignment, and $\alpha \geq 0$ is a weighting parameter. This formulation reduces to the classical welfare problem when $T \equiv 0$ and arises naturally in applications such as electric vehicle routing, urban logistics~\cite{beckmann1956studies,sheffi1985urban}, and bandwidth-constrained feature acquisition~\cite{Yang2003}.

The resulting optimization problem is substantially more challenging than the standalone welfare formulation for two interconnected reasons. \emph{First}, the feasible set is the product
$\mathcal{I} \times \Omega_1 \times \cdots \times \Omega_n$ of allocation and path-selection decisions, yielding a multi-level combinatorial problem that is generally NP-hard~\cite{ahmed2003multi}. \emph{Second}, and more fundamentally, the routing cost
$T(\mathscr{S},\mathscr{R})$ introduces coupling \emph{across} providers: when routes share edges, the marginal cost incurred by one provider depends on the simultaneous load induced by the others, a hallmark of \emph{supermodularity}~\cite{beckmann1956studies,sheffi1985urban}. This cross-provider dependence destroys the additive separability underlying the standard welfare problem and prevents a direct application of classical approximation machinery.

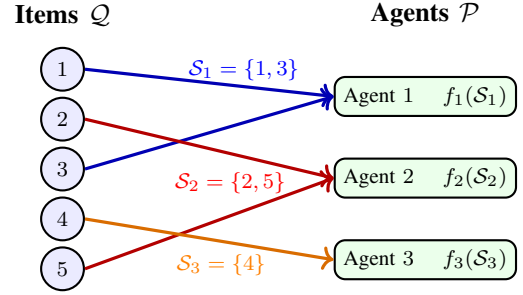
\begin{figure}[t]
\centering
    \begin{tikzpicture}[scale=0.6, every node/.style={font=\footnotesize}]
    \tikzstyle{item}=[circle, draw, thick, minimum size=2mm, fill=blue!8]
    \tikzstyle{agent}=[rounded corners, draw, thick, minimum width=0.7cm, minimum height=0.1cm, fill=green!8]
    \tikzstyle{assign1}=[->, very thick, blue!70!black]
    \tikzstyle{assign2}=[->, very thick, red!70!black]
    \tikzstyle{assign3}=[->, very thick, orange!85!black]
    
    \node[font=\bfseries] at (0,4.4) {Items $\mathcal{Q}$};
    \node[font=\bfseries] at (8,4.4) {Agents $\mathcal{P}$};
    
    \node[item] (u1) at (0,3.2) {$1$};
    \node[item] (u2) at (0,2.1) {$2$};
    \node[item] (u3) at (0,1.0) {$3$};
    \node[item] (u4) at (0,-0.1) {$4$};
    \node[item] (u5) at (0,-1.2) {$5$};
    
    \node[agent] (a1) at (8,2.6) {Agent $1$ \quad $f_1(\mathcal{S}_1)$};
    \node[agent] (a2) at (8,0.8) {Agent $2$ \quad $f_2(\mathcal{S}_2)$};
    \node[agent] (a3) at (8,-1.0) {Agent $3$ \quad $f_3(\mathcal{S}_3)$};
    
    \draw[assign1] (u1.east) -- (a1.west);
    \draw[assign1] (u3.east) -- (a1.west);
    
    \draw[assign2] (u2.east) -- (a2.west);
    \draw[assign2] (u5.east) -- (a2.west);
    
    \draw[assign3] (u4.east) -- (a3.west);
    
    \node[align=left] at (4.0,3.2) {\textcolor{blue}{$\mathcal{S}_1=\{1,3\}$}};
    \node[align=left] at (3.7,0.7) {\textcolor{red}{$\mathcal{S}_2=\{2,5\}$}};
    \node[align=left] at (3.5,-1.1) {\textcolor{orange}{$\mathcal{S}_3=\{4\}$}};
    
    \end{tikzpicture}
    \caption{Illustration of the welfare maximization problem, where items are partitioned among agents so as to maximize the sum of agent valuations.}
    \label{fig:welfare_problem}
\end{figure}

A natural attempt to recover tractability is to mimic the partition matroid lifting of the welfare problem by enlarging the ground set to include provider--item--path triples,
\[
X' = \mathcal{P} \times \mathcal{Q} \times \bigcup_{i=\mathcal{P}} \Omega_i,
\]
and then imposing a generalized partition matroid. However, this approach also presents two structural drawbacks. Computationally, the number of simple paths in $\mathcal{G}$ grows exponentially with $|\mathcal{V}|$, rendering direct enumeration intractable~\cite{chekuri2005recursive}. On the theoretical side, supermodularity of $T$ induces positive cross-partial differences on shared network resources, as the marginal transportation cost of routing one provider's items increases as other providers load the same infrastructure. This interdependence violates the independent structure required by matroid-based formulations \cite{Whitney1935}.

Alternatively, when only addressing computational complexity, the literature has largely treated the two subproblems in isolation. The pure \emph{allocation} problem under nonlinear utilities is well understood through submodular welfare theory~\cite{jin2021usmmc}, whereas the pure \emph{routing} problem is classically addressed by shortest-path methods on directed graphs~\cite{chekuri2005recursive}. Joint formulations, however, remain limited. The closest related work~\cite{zhang2016submodular} considers submodular objectives under knapsack-type routing constraints and establishes a $\frac{1}{2}(1-1/e)$-approximation. However, it remains fundamentally limited in two aspects. First, routing is treated purely as a \emph{feasibility constraint} rather than as a decision variable to be jointly optimized, which essentially differs from our setting. Additionally, the model is restricted to \emph{additive routing costs}, thereby excluding congestion-dependent or capacity-weighted regimes. As a result, it does not capture the coupled allocation--routing structure we study, and a unifying framework accommodating linear and nonlinear routing with joint optimization remains missing.

The key structural observation of this paper is that, although each local utility $f_i$ is submodular
(Assumption~\ref{assump:submodularity_utility}) and the routing cost $T$ is supermodular
(Assumption~\ref{assump:supermodularity_cost}), their difference remains \emph{submodular} in the allocation variable; see Lemmas~\ref{lem:sub_plus_sub} and~\ref{lem:sub_minus_super}. While this follows from standard closure properties, it reveals that routing-induced coupling does not destroy the submodular structure of the allocation problem. This observation enables us to retain tractable optimization despite the loss of separability.

Building on this structural insight, this paper introduces a joint allocation and routing framework that covers both additive and non-additive transportation costs. We first establish that the coupled formulation preserves submodularity in the allocation variable and derive a hierarchical conditioned decomposition. We then characterize the modular routing regime, where routing becomes additively separable across providers, and use this case to motivate an efficient initialization mechanism. Crucially, we model the supermodular routing effects as a bounded perturbation of the modular case, which allows us to quantify the impact of routing-induced coupling on solution quality. Finally, we propose a hierarchical decomposition-based algorithm with convergence guarantees and derive approximation bounds that explicitly depend on the deviation from modular routing.

The remainder of this paper is organized as follows. Section~\ref{sec:statement}
formalizes the routing and allocation model as a set-function optimization problem. Section~\ref{sec:problem_structure} establishes the relevant submodular properties and derives the hierarchical decomposition, including decoupling under modular routing costs. Section~\ref{sec:methods} presents the proposed optimization framework and its theoretical guarantees. Section~\ref{sec:case} provides numerical results, and Section~\ref{sec:conclusion} concludes the paper.

\begin{figure}[t]
\centering
\begin{tikzpicture}[scale=1.0, every node/.style={font=\small}]

\tikzstyle{vtx}=[circle, draw, thick, minimum size=7mm, fill=blue!8]
\tikzstyle{edge}=[draw=black!60, thick]
\tikzstyle{routeA}=[draw=blue!75!black, very thick, ->]
\tikzstyle{routeB}=[draw=red!75!black, very thick, ->]
\tikzstyle{quadopt}=[draw=green!50!black, very thick, dashed, ->]
\tikzstyle{cost}=[fill=white, inner sep=1pt]

\node[vtx, fill=green!15] (s1) at (0.5,2.2) {$u_1$};
\node[vtx, fill=green!15] (s2) at (0.5,0.8) {$u_2$};

\node[vtx] (a) at (2.5,3.0) {$a$};
\node[vtx] (b) at (2.5,1.5) {$b$};
\node[vtx] (c) at (2.5,0.0) {$c$};

\node[vtx] (d) at (4.5,2.4) {$d$};
\node[vtx] (e) at (4.5,0.6) {$e$};

\node[vtx, fill=red!15] (t) at (6.5,1.5) {$v$};

\draw[edge] (s1) -- (a) node[midway, above left, cost] {$2$};
\draw[edge] (s1) -- (b) node[midway, above, cost] {$1$};

\draw[edge] (s2) -- (b) node[midway, below, cost] {$1$};
\draw[edge] (s2) -- (c) node[midway, below left, cost] {$2$};

\draw[edge] (a) -- (d) node[midway, above, cost] {$2$};
\draw[edge] (a) -- (b) node[midway, left, cost] {$2$};

\draw[edge] (b) -- (d) node[midway, above right, cost] {$2$};
\draw[edge] (b) -- (e) node[midway, below right, cost] {$1$};

\draw[edge] (c) -- (e) node[midway, below, cost] {$2$};
\draw[edge] (d) -- (t) node[midway, above right, cost] {$2$};
\draw[edge] (e) -- (t) node[midway, below right, cost] {$1$};

\draw[edge] (d) -- (e) node[midway, right, cost] {$1$};

\draw[routeA] (s1) -- (b);
\draw[routeA] (b) -- (e);
\draw[routeA] (e) -- (t);

\draw[routeB] (s2) -- (b);
\draw[routeB] (b) -- (e);
\draw[routeB] (e) -- (t);

\draw[line width=5pt, draw=orange!40, opacity=0.45] (b) -- (e);
\draw[line width=5pt, draw=orange!40, opacity=0.45] (e) -- (t);

\draw[quadopt] (s1) -- (b);
\draw[quadopt] (b) -- (d);
\draw[quadopt] (d) -- (t);

\draw[quadopt] (s2) -- (b);
\draw[quadopt] (b) -- (e);
\draw[quadopt] (e) -- (t);
\end{tikzpicture}
\caption{Comparison between modular and quadratic routing costs. Solid blue/red routes illustrate the overlapping solution preferred under modular costs, while dashed green routes illustrate the optimal joint solution under quadratic costs, where one route deviates to reduce overlap.}
\label{fig:quadratic_routing_overlap}
\end{figure}
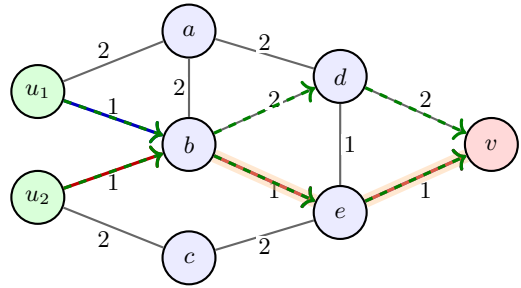

\section{Problem Statement}
\label{sec:statement}

We consider a discrete resource allocation problem in which a central
network authority must optimally distribute a set of $m$ discrete items $\mathcal{Q} = \{1,\cdots,m\}$,
requested by a single client and distributed among $\kappa$ providers. The providers and
client are interconnected via a transportation network modeled as a directed
weighted graph $\mathcal{G} = (\mathcal{V}, \mathcal{E}, w)$, where
$\mathcal{V} = \{1, \ldots, N\}$ is the set of nodes,
$\mathcal{E} \subseteq \mathcal{V} \times \mathcal{V}$ is the set of directed
edges, and $w : \mathcal{E} \rightarrow \mathbb{R}_{\geq 0}$ is the edge
weight function assigning a non-negative cost to each edge
$(i,j) \in \mathcal{E}$ \cite{Kalaivani2024}. Intuitively, the weight function $w$ encodes travel costs, latency, distance, or capacity, see Section~\ref{sec:case}.

Without loss of generality, providers are located at nodes
$\{1, \ldots, \kappa\} \subset \mathcal{V}$, and the client resides at a
designated node $c\in\mathcal{V}$. Hence, each provider $i\in\mathcal{P}=\{1,\cdots,\kappa\}$ is associated with a finite, non-empty set of directed simple paths $(i,\cdots,c)\in\mathscr{R}_i*$ in $\mathcal{G}$. Hereafter, a \emph{feasible allocation} is a partition $\mathscr{S}=(\mathcal{S}_1,\cdots,\mathcal{S}_\kappa)$ of $\mathcal{Q}$ where $\mathcal{S}_i\subseteq\mathcal{Q}$ denotes the (possibly empty) subset of items assigned to provider $i\in\mathcal{P}$. 

Each provider $i$ is associated with a local utility function
$f_i : 2^{\mathcal{Q}} \rightarrow \mathbb{R}$, where $2^{\mathcal{Q}}$
denotes the power set of $\mathcal{Q}$, mapping any item subset
$\mathcal{S}_i \subseteq \mathcal{Q}$ to a real-valued score representing the utility of that allocation for provider $i$. This formulation is sufficiently general to encode provider-side costs through the
sign convention of $f_i$, see Definitions~\ref{assum:normality} and \ref{assum:monotone}.

\begin{defn}[Normality]
    \label{assum:normality}
    A function $f$ is normal if $f(\emptyset)=0$.
\end{defn}

\begin{defn}[Increasing Monotonicity]
    \label{assum:monotone}
    A function $f$ is monotone increasing if for any $\mathscr{S}_1\subseteq\mathscr{S}_2\subseteq2^{X}$ we have $ f(\mathscr{S}_2)\geq f(\mathscr{S}_1)$. 
\end{defn}

In addition to the local utilities, the transportation network $\mathcal{G}$ imposes a routing cost on any given allocation. For a feasible allocation $\mathscr{S}=(\mathcal{S}_1,\cdots,\mathcal{S}_\kappa)$ and a route assignment $\mathscr{R}=(\mathcal{R}_1,\cdots,\mathcal{R}_\kappa)$, where $\mathcal{R}_i\in\mathscr{R}_i*$, the total routing cost is denoted as $T(\mathscr{S},\mathscr{R})$. Then, the optimal resource allocation problem is formulated as follows:
\begin{subequations}
\begin{align}
    \max_{\mathscr{S},\mathscr{R}} \quad & \sum_{i=1}^\kappa f_i(\mathcal{S}_i) - \alpha \cdot T(\mathscr{S},\mathscr{R}) \label{eqn:utility}\\
    \text{subject to} \quad & \mathcal{S}_i \subset \mathcal{Q}, \qquad~~ \forall i \in \mathcal{P} \label{eqn:constrain_1}\\
    & \mathcal{S}_i \cap \mathcal{S}_j = \emptyset, \quad \forall i \neq j, ~(i,j)\in\mathcal{P} \label{eqn:constrain_2}\\
    & \bigcup_{i=1}^\kappa \mathcal{S}_i = \mathcal{Q}, \label{eqn:constrain_3} \\
    & \mathcal{R}_i\in\mathscr{R}_i*, \qquad~~ \forall i \in\mathcal{P}, \label{eqn:constrain_4}
\end{align}
\label{eqn:problem2}
\end{subequations}
where $\mathscr{S}$ is the desired item allocation assigning each of $\mathcal{Q}$ to exactly one provider in $\mathcal{P}$, and $\mathscr{R}$ are the corresponding route assignments over the edges $\mathcal{E}$ used for transportation over the network. Constraints~\eqref{eqn:constrain_1}--\eqref{eqn:constrain_3} jointly enforce that $\mathscr{S}$ is a valid partition of $\mathcal{Q}$ such that every item is assigned to exactly one provider, and the union of all assigned disjoint subsets recovers $\mathcal{Q}$. The use of $\subseteq$ in~\eqref{eqn:constrain_1},
rather than strict inclusion $\subset$, explicitly permits empty assignments,
i.e., a provider may receive no items. Constraint~\eqref{eqn:constrain_4}
ensures that each provider selects exactly one path from its admissible path
set $\mathscr{R}_i$ for transporting its allocated items to the
client.\footnote{The original constraint $|\mathscr{R} \cap \mathscr{R}_i*|=1$
has been restated as $\mathcal{R}_i \in \mathscr{R}_i*$ for notational clarity
and consistency with the tuple definition
$\mathscr{R} = (\mathcal{R}_1, \ldots, \mathcal{R}_\kappa)$.}

Importantly, the routing cost $T(\mathscr{S},\mathscr{R})$ is not separable across providers, as shared edges induce interactions that depend on the joint routing decisions, see Assumption~\ref{assump:supermodularity_cost}. In particular, we consider routing costs that exhibit supermodular behavior with respect to the allocation, capturing congestion and shared-resource effects. As a result, the resulting optimization problem is inherently coupled and cannot be decomposed into independent allocation and routing subproblems. It will be explicitly defined in Section~\ref{sec:case}.

\begin{assump}
    \label{assump:submodularity_utility}
    The local objective function $f_i(\mathcal{S}):2^\mathcal{Q}\rightarrow\mathbb{R}_{\geq 0}$ associated with a subset of items $\mathcal{S}\subset\mathcal{Q}$ representing the benefit of allocating $\mathcal{S}$ in provider $i$ is submodular as it exhibits diminishing returns with the increase of the allocated items~\cite{FeigeVondrak2006allocation}. 
\end{assump}

\begin{assump}
    \label{assump:supermodularity_cost}
    The routing cost $T(\mathscr{S};\mathscr{R})$ across providers is individually supermodular with respect to $\mathscr{S}$ as it exhibits the increasing returns property with the increase of transported items \cite{supermodulartransport}. 
\end{assump}

\noindent\emph{Notation and definitions}: For a discrete ground set $\mathcal{P}$, $2^{\mathcal{P}}$ is its power set, the set that contains all the subsets of $\mathcal{P}$. A set function $f: 2^{\mathcal{P}} \rightarrow \mathbb{R}_{ \geq 0}$ is \textit{submodular} if and only if for any $\mathcal{P}_2 \subseteq \mathcal{P}_1 \subseteq {\mathcal{P}}$ and for all $p \in \mathcal{P} \setminus \mathcal{P}_1$ we have that
\begin{equation}
 \label{eqn:diminishing_returns}
 f(\mathcal{P}_2 \cup \{ p \}) - f(\mathcal{P}_2) \geq f(\mathcal{P}_1 \cup \{ p \}) - f(\mathcal{P}_1)
\end{equation}
Function $f$ is normal if $f(\emptyset)=0$ and  is monotone increasing if  for any $\mathcal{P}_1\subset\mathcal{P}_2\subset\mathcal{P}$ we have $ f(\mathcal{P}_2)\geq f(\mathcal{P}_1)$. For any $p\in\mathcal{P} $ and any $\mathcal{P}\subset\mathcal{P}$, $\Delta (p|\mathcal{P})=f(\mathcal{P} \cup \{ p\}) - f(\mathcal{P})$ is the marginal gain of adding $p$ to the set $\mathcal{P}$. 


\section{Combinatorial Structure Of Joint Allocation And Routing}
\label{sec:problem_structure}

By Assumptions~\ref{assump:submodularity_utility} and \ref{assump:supermodularity_cost}, as mentioned in Section~\ref{sec::intro}, the allocation and routing decisions are intrinsically coupled, and a direct decoupling would generally compromise optimality. In this section, we characterize the combinatorial structure of the joint problem and show that, despite this coupling, the objective preserves submodularity with respect to the allocation variable. This property is the key enabler for the approximation guarantees developed in Section~\ref{sec:methods}.

\subsection{Allocation Structure and Routing Coupling}

We begin by reformulating Problem~\eqref{eqn:utility} as
\begin{subequations}
\begin{align}
    \max_{\mathscr{S},\mathscr{R}} \quad & f(\mathscr{S},\mathscr{R}) = F(\mathscr{S}) - \alpha \cdot T(\mathscr{S},\mathscr{R}) \label{eqn:utility2}\\
    \text{subject to} \quad & \mathscr{S} \in \mathcal{I}, \label{eqn:constrain_1_2}\\
    & \mathcal{R}_i \in \mathscr{R}_i, \qquad \forall i \in \mathcal{P}, \label{eqn:constrain_2_2}
\end{align}
\label{eqn:problem}
\end{subequations}
where $F(\mathscr{S}) = \sum_{i=1}^\kappa f_i(\mathcal{S}_i)$ and
{\scriptsize
$$
\mathcal{I}=\Big\{ \mathscr{S}=(\mathcal{S}_1,\cdots,\mathcal{S}_\kappa) \;\Big|\; \mathcal{S}_i\subseteq\mathcal{Q},\;
\mathcal{S}_i\cap\mathcal{S}_j=\emptyset~\forall i\neq j,\;
\bigcup_{i=1}^{\kappa}\mathcal{S}_i=\mathcal{Q} \Big\}.
$$
}

The allocation constraints admit a classical partition matroid representation. Consider the ground set $X = \mathcal{P} \times \mathcal{Q}$, where each pair $(i,q)$ denotes assigning item $q$ to provider $i$. The independent sets are given by
\[
\mathcal{I}_{\mathcal{P}\times\mathcal{Q}}
=
\left\{
A \subseteq \mathcal{P} \times \mathcal{Q}
\;\middle|\;
|A \cap (\mathcal{P} \times \{q\})| \le 1,\;\forall q \in \mathcal{Q}
\right\},
\]
which enforce that each item is assigned to at most one provider. This defines a partition matroid over $X$.

In contrast, the routing component does not, in general, induce a matroid structure. The feasible routing set $\mathscr{R}_i$, consisting of all source-to-client paths for each provider, forms a combinatorial family that does not satisfy the exchange axiom. Therefore, the joint problem should be interpreted as a coupling between a matroid-constrained allocation problem and a combinatorial routing subproblem.

\subsection{Submodularity of the Joint Objective}

We now establish the key structural property of the objective function.

\begin{lem}
    Given a submodular function $f$ and another submodular function $g$, their sum $h = f+g$ is submodular.
    \label{lem:sub_plus_sub}
\end{lem}
\begin{proof}
    The proof is standard, see~\cite{Bach2013submodular}.
\end{proof}

\begin{lem}
    Given a submodular function $f$ and a supermodular function $g$, the function $h = f-g$ is submodular.
    \label{lem:sub_minus_super}
\end{lem}
\begin{proof}
    Since $-g$ is submodular, the result follows from Lemma~\ref{lem:sub_plus_sub}.
\end{proof}

\begin{thm}[Submodularity in the allocation variable]
    \label{thm:problem_submodularity}
    For any fixed routing assignment $\mathscr{R}$, the function
    \[
        f_{\mathscr{R}}(\mathscr{S}) = F(\mathscr{S}) - \alpha T(\mathscr{S},\mathscr{R})
    \]
    is submodular in $\mathscr{S}$.
\end{thm}
\begin{proof}
    By Assumption~\ref{assump:submodularity_utility}, $F(\mathscr{S})$ is submodular in $\mathscr{S}$. By Assumption~\ref{assump:supermodularity_cost}, $T(\mathscr{S},\mathscr{R})$ is supermodular in $\mathscr{S}$ for any fixed $\mathscr{R}$. Therefore, $-\alpha T(\mathscr{S},\mathscr{R})$ is submodular, and the result follows from Lemma~\ref{lem:sub_plus_sub}.
\end{proof}

Note that Theorem~\ref{thm:problem_submodularity} holds regardless of whether $T(\mathscr{S},\mathscr{R})$ is modular or supermodular. In the modular case, the routing term becomes separable and independent of $\mathscr{S}$, while in the supermodular regime it introduces coupling across providers, yet preserves submodularity of the objective with respect to $\mathscr{S}$.

\subsection{Hierarchical Decomposition of Allocation and Routing}

The joint optimization problem couples two decision variables: the allocation
$\mathscr{S}$ and the routing $\mathscr{R}$. While the allocation lies in a
partition matroid domain, the routing corresponds to a combinatorial path
selection problem. Despite this structural difference, the objective admits
a natural hierarchical decomposition.

In particular, for any fixed allocation $\mathscr{S}$, the routing subproblem
reduces to minimizing $T(\mathscr{S},\mathscr{R})$ over feasible paths. This
induces a reduced objective over $\mathscr{S}$, which forms the basis of the
proposed algorithm.

\begin{lem}[Finiteness of $\mathscr{R}$]
    \label{lem:R_finitness}
    Given a finite graph $\mathcal{G}$ and two nodes $u$ and $v$, the number of simple paths from $u$ to $v$ is finite.
\end{lem}

\begin{proof}
Since $\mathcal{G}$ has a finite number of nodes, any simple path can visit each node at most once, and therefore has length at most $|\mathcal{V}|-1$. Hence, the number of such paths is finite.
\end{proof}

\begin{lem}
    \label{lem:R_defined}
    For any fixed allocation $\mathscr{S}\in\mathcal{I}$, the routing problem is well-defined.
\end{lem}

\begin{proof}
Assuming the graph $\mathcal{G}$ is connected, each provider node has at least one feasible path to the client node. Hence, a feasible routing $\mathscr{R}$ exists.
\end{proof}

\begin{thm}\label{thm:hierarchical}
For any $\alpha > 0$, the optimization
Problem~\eqref{eqn:problem} is equivalent to the single-variable problem
\begin{equation}\label{eq:reduced}
    \max_{\mathscr{S} \in \mathcal{I}}
    \Bigl( F(\mathscr{S}) - \alpha\, T\bigl(\mathscr{S},\,
    \mathscr{R}^\star(\mathscr{S})\bigr) \Bigr),
\end{equation}
where, for each fixed $\mathscr{S} \in \mathcal{I}$, the optimal
routing is
\begin{equation}\label{eq:opt_routing}
    \mathscr{R}^\star(\mathscr{S})
    \;=\;
    \operatorname*{arg\,min}_{\substack{\mathcal{R}:\,
    \mathcal{R}_i \in \mathscr{R}_i,\, \forall\, i \in \mathcal{P}}}
    \; T(\mathscr{S},\mathscr{R}).
\end{equation}
\end{thm}

\begin{proof}
For any fixed allocation $\mathscr{S}$, define
$\mathscr{R}^\star(\mathscr{S}) = \arg\min_{\mathscr{R}} T(\mathscr{S},\mathscr{R}).$ Then, for any feasible pair $(\mathscr{S},\mathscr{R})$, $T(\mathscr{S},\mathscr{R}^\star(\mathscr{S})) \le T(\mathscr{S},\mathscr{R}),$
which implies
$$f(\mathscr{S},\mathscr{R}) = F(\mathscr{S}) - \alpha T(\mathscr{S},\mathscr{R}) \le F(\mathscr{S}) - \alpha T(\mathscr{S},\mathscr{R}^\star(\mathscr{S})),$$
taking the maximum over $\mathscr{S}$ yields the result.
\end{proof}

This decomposition naturally motivates a coordinate-wise optimization strategy, where allocation and routing are updated iteratively.

\section{Methodology}
\label{sec:methods}

Building on the hierarchical decomposition established in Theorem~\ref{thm:hierarchical}, we propose the \textbf{W}elfare-based \textbf{Hi}erarchical \textbf{R}outing A\textbf{l}gorithm (\texttt{WHIRL}), Algorithm~\ref{alg:hier_greedy}. The proposed method alternates between the two coupled subproblems: routing and allocation. In doing so, it exploits the problem structure to reduce computational complexity while preserving approximation guarantees.

The method follows a classical coordinate-wise optimization scheme. Starting from an initial allocation, the routing subproblem in line~$3$ is solved for the current assignment. Then, fixing the resulting routes, the allocation subproblem in line~$5$ is solved. This alternating process is repeated until convergence, as established next.

\begin{lem}[Convergence]
    \label{lem:bounded_improvement}
    Given Algorithm~\ref{alg:hier_greedy} applied to Problem~\eqref{eqn:problem}, the sequence of iterates generated by \texttt{WHIRL} converges in a finite number of stages.
\end{lem}
\begin{proof}
    Let $(\mathscr{S}_i,\mathscr{R}_i)$ be the current iterate. By construction of Algorithm~\ref{alg:hier_greedy}, the routing update computes
    $$\mathscr{R}_{i+1}\in\arg\min_{\mathscr{R}} T(\mathscr{S}_i,\mathscr{R}),$$
    and therefore
    $$F(\mathscr{S}_i)-\alpha T(\mathscr{S}_i,\mathscr{R}_i) \le F(\mathscr{S}_i)-\alpha T(\mathscr{S}_i,\mathscr{R}_{i+1}).$$
    Next, fixing $\mathscr{R}_{i+1}$, the allocation update computes $\mathscr{S}_{i+1}$ so that
    $$F(\mathscr{S}_i)-\alpha T(\mathscr{S}_i,\mathscr{R}_{i+1})
        \le F(\mathscr{S}_{i+1})-\alpha T(\mathscr{S}_{i+1},\mathscr{R}_{i+1}).$$
    Combining both inequalities yields
    $f(\mathscr{S}_i,\mathscr{R}_i) \le f(\mathscr{S}_{i+1},\mathscr{R}_{i+1})$. Hence, the objective sequence is monotonically non-decreasing.

    Finally, since the allocation space is finite and, by Lemma~\ref{lem:R_finitness}, the routing space is also finite, the number of feasible pairs $(\mathscr{S},\mathscr{R})$ is finite. Therefore, a monotone sequence of objective values generated over a finite feasible set must terminate after finitely many iterations.
\end{proof}

Observe that \texttt{WHIRL} alternates between two optimization subproblems induced by Theorem~\ref{thm:hierarchical}. For fixed routing, the allocation step reduces to a submodular maximization problem over the partition matroid constraint, which can be addressed by \texttt{SGA} with $\frac{1}{2}$-approximation guarantees, Section~\ref{sec::intro}. For fixed allocation, the routing step becomes a combinatorial minimization problem, which under modular costs reduces to shortest-path computation and under supermodular costs requires additional approximation analysis. To fully characterize the guarantees of Algorithm~\ref{alg:hier_greedy}, we next study the routing subproblem in detail.

\begin{algorithm}[t]
  \caption{Welfare-based Hierarchical Routing Algorithm}
  \label{alg:hier_greedy}
  \begin{algorithmic}[1]
    \Require $2$ Ground sets $\mathcal{P}$ and $\mathcal{Q}$, graph $\mathcal{G}$
    \Ensure Allocation $(\mathcal{S}_1,\ldots,\mathcal{S}_\kappa)$ and routing $(\mathcal{R}_1,\ldots,\mathcal{R}_\kappa)$

    \State Randomly initialize $\mathscr{S}$ 
    \While{not convergence}
       \State Find $\bar{\mathscr{R}}(\mathscr{S})\gets \arg\min_{\mathscr{R}\in\mathscr{R}} T(\mathscr{S},\mathscr{R})$
       \State $\mathscr{R}\gets\bar{\mathscr{R}}(\mathscr{S})$
       \State Find $\bar{\mathscr{S}}(\mathscr{R}) \gets \arg\max_{\mathscr{S}\in\mathcal{I}_{\mathcal{P}\times\mathcal{Q}}} F(\mathscr{S})-\alpha\cdot T(\mathscr{S},\mathscr{R})$
       \State $\mathscr{S}\gets \bar{\mathscr{S}}(\mathscr{R})$
    \EndWhile
    \State \Return $\mathscr{S},\mathscr{R}$
  \end{algorithmic}
\end{algorithm}

\subsection{Problem~\eqref{eq:opt_routing}: Routing under Supermodular Costs}

We now analyze the routing subproblem induced by Theorem~\ref{thm:hierarchical}. For a fixed allocation $\mathscr{S}$, the routing problem consists of selecting, for each provider, a path minimizing the routing cost $T(\mathscr{S},\mathscr{R})$.

\textit{Modular routing.} We begin with the classical case where routing costs are additive across edges, demonstrating that, under modular routing, Problem~\eqref{eq:opt_routing} is decoupled from Welfare Problem yielding in independent shortest-path finding.

\begin{lem}
    \label{lem:single_ssp}
    Consider a single provider located at node $i$ and a client at node $j$. If the routing cost is additive, i.e., $T(\mathcal{R}) = \sum_{e\in\mathcal{R}} w_e$,
    with $w_e \ge 0$, then an optimal route $\mathcal{R}^\star$ is given by the shortest path from $i$ to $j$.
\end{lem}

\begin{proof}
Since the cost of a route is the sum of nonnegative edge weights, minimizing $T(\mathcal{R})$ is equivalent to finding a shortest path in the graph. Standard shortest-path optimality implies the result.
\end{proof}

\begin{thm}
    \label{thm:additive_optimality}
    If the routing cost is additive,
    $T(\mathscr{S},\mathscr{R}) = \sum_{i=1}^\kappa \sum_{e\in \mathcal{R}_i} w_e(\mathscr{S})$,
    then the routing problem decomposes across providers, and each $\mathcal{R}_i$ is given by a shortest path from provider $i$ to the client. In this case, routing and allocation decisions are separable.
\end{thm}

\begin{proof}
Under additive costs, $T(\mathscr{S},\mathscr{R})$ is separable across providers. Therefore, minimizing $T$ reduces to solving $\kappa$ independent shortest-path problems, one for each provider. Hence, routing decisions do not interact and can be solved independently of allocation.
\end{proof}


\textit{Supermodular routing as perturbation.}
We now consider the case where routing costs exhibit supermodular effects through interactions across providers. We model this as a bounded perturbation of the modular case, see Figure~\ref{fig:curvature_modularity}.

\begin{defn}[Routing Modular Deviation]
\label{def:routing_curvature}
Let $\widetilde{T}$ be a modular surrogate of the routing cost and $T$ the true (supermodular) routing cost. We define the routing modular deviation $\delta$ as the smallest constant such that, for any feasible allocation $\mathscr{S}$ and any route $\mathcal{R}$,
$$ \delta = \max_{e\in(\mathscr{S},\mathcal{R})} w_e(\mathscr{S},\mathcal{R}) -  \widetilde{w}_e(\mathscr{S},\mathcal{R}).$$
\boxend
\end{defn}

Assume that for each edge $e$,
\begin{equation}
    \label{eqn:perturbation}
    w_e(\mathscr{S}) = \widetilde{w}_e(\mathscr{S}) + \Delta_e(\mathscr{S}),
    \quad \text{with } 0 \le \Delta_e(\mathscr{S}) \le \delta, 
\end{equation}
where $w_e$ is the cost of $e$ under supermodular regime, $\widetilde{w}_e$ is the cost under modular regime and $\Delta_e$ is the difference between both utilities. Then, for any route $\mathcal{R}$,
\begin{equation}
    \label{eqn:perturbation2}
    T(\mathscr{S},\mathcal{R}) = \sum_{e\in\mathcal{R}} w_e(\mathscr{S}) \le \widetilde{T}(\mathscr{S},\mathcal{R}) + \delta |\mathcal{R}|,
\end{equation}
where $\widetilde{T}$ denotes the modular surrogate of routing cost. 

\begin{prop}[Perturbation bound]
    \label{prop:bounded_distance}
    Let $\widetilde{\mathcal{R}}^\star$ be an optimal route under the modular surrogate cost $\widetilde{T}$, and let $\mathcal{R}^\star$ be optimal under the supermodular cost $T$. Then, for any fixed $\mathscr{S}$,
    $$\widetilde{T}(\mathscr{S},\widetilde{\mathcal{R}}^\star)
        \le
        T(\mathscr{S},\mathcal{R}^\star)
        \le
        \widetilde{T}(\mathscr{S},\widetilde{\mathcal{R}}^\star) + \delta |\widetilde{\mathcal{R}}^\star|.$$
\end{prop}
\vspace{1mm}
\begin{proof}
By optimality of $\widetilde{\mathcal{R}}^\star$ under $\widetilde{T}$,
$T(\mathscr{S},\mathcal{R}^\star) \le T(\mathscr{S},\widetilde{\mathcal{R}}^\star)$. Applying the perturbation bound \eqref{eqn:perturbation2} to $\widetilde{\mathcal{R}}^\star$, $T(\mathscr{S},\widetilde{\mathcal{R}}^\star)
\le
\widetilde{T}(\mathscr{S},\widetilde{\mathcal{R}}^\star)
+ \delta |\widetilde{\mathcal{R}}^\star|,$ which yields the upper bound. Then, for the lower bound, since $\widetilde{\mathcal{R}}^\star$ is optimal under $\widetilde{T}$, $\widetilde{T}(\mathscr{S},\widetilde{\mathcal{R}}^\star)
\le
\widetilde{T}(\mathscr{S},\mathcal{R}^\star)
\le
T(\mathscr{S},\mathcal{R}^\star),$
where the last inequality follows from $T \ge \widetilde{T}$, concluding the proof.
\end{proof}

\begin{prop}[Routing gap under perturbation]
    \label{prop:join_bounded_distance}
    Let $\widetilde{\mathscr{R}}^\star = (\widetilde{\mathcal{R}}_1^\star,\dots,\widetilde{\mathcal{R}}_\kappa^\star)$ be optimal under the modular surrogate cost $\widetilde{T}$, and let $\mathscr{R}^\star = (\mathcal{R}_1^\star,\dots,\mathcal{R}_\kappa^\star)$ be optimal under the perturbed cost $T$. Then, for any fixed $\mathscr{S}$,
    \[
        \widetilde{T}(\mathscr{S},\mathscr{R}^\star)
        \le
        \widetilde{T}(\mathscr{S},\widetilde{\mathscr{R}}^\star) + \kappa \delta |\bar{\mathcal{R}}^\star|,
    \]
    where $|\bar{\mathcal{R}}^\star| = \max_{i\in\{1,\ldots,\kappa\}} |\mathcal{R}^\star_i|$.
\end{prop}

\begin{proof}
Applying Proposition~\ref{prop:bounded_distance} to each provider and summing over the $\kappa$ routes yields the result.
\end{proof}

\subsection{Optimality Guarantees}

After characterizing optimality gap for both processes, welfare and routing, let us prove optimality guarantees of \texttt{WHIRL}. 

\begin{thm}[Optimality of Algorithm~\ref{alg:hier_greedy}]
    \label{thm:optimality}
    Let
    $$f(\mathscr{S},\mathscr{R}) \;=\; F(\mathscr{S})-\alpha\,T(\mathscr{S},\mathscr{R}),$$
    where \(T\) is the true supermodular routing cost. Let
    \((\mathscr{S}^\star,\mathscr{R}^\star)\) be an optimal solution of
    Problem~\eqref{eqn:problem}, and let \((\bar{\mathscr{S}},\bar{\mathscr{R}})\)
    be the solution returned by Algorithm~\ref{alg:hier_greedy}. Assume that for any fixed routing \(\mathscr{R}\), the welfare subproblem admits the \(1/2\)-approximation of \cite{FeigeVondrak2006allocation}, and that the supermodular perturbation satisfies,
    $$T(\mathscr{S},\mathscr{R})
        \;\leq\;  \widetilde{T}(\mathscr{S},\mathscr{R})+\kappa\delta\,|\bar{\mathcal{R}}|,$$
    where \(|\bar{\mathcal{R}}|\) is an upper bound on the length of the longest route in the assigned routing. Then
    $$f(\bar{\mathscr{S}},\bar{\mathscr{R}})
        \;\geq\;
        \frac{1}{2}\,f(\mathscr{S}^\star,\mathscr{R}^\star)
        \;-\;
        \frac{\alpha}{2}\,\kappa\delta\,|\bar{\mathcal{R}}^\star|.$$
\end{thm}
\vspace{1mm}
\begin{proof}
    Let \((\mathscr{S}_0,\mathscr{R}_0)\) denote the initialization of Algorithm~\ref{alg:hier_greedy},
    where \(\mathscr{R}_0\) is obtained by Dijkstra's algorithm~\cite{DJL13} on the modular approximation \(T\) which is known to be optimal \cite{dijkstra2}.
    By Lemma~\ref{lem:bounded_improvement}, the iterations of Algorithm~\ref{alg:hier_greedy}
    are monotone, hence the final solution satisfies
    $$f(\bar{\mathscr{S}},\bar{\mathscr{R}})
        \;\geq\;
        f(\mathscr{S}_1,\mathscr{R}_0),$$
    where \(\mathscr{S}_1\) is the allocation obtained after the first welfare update. Now fix \(\mathscr{R}_0\). By the welfare approximation guarantee of
    \cite{FeigeVondrak2006allocation},
    $$F(\mathscr{S}_1)-\alpha\,\widetilde{T}(\mathscr{S}_1,\mathscr{R}_0)
        \;\geq\;
        \frac{1}{2}
        \max_{\mathscr{S}}
        \big(
            F(\mathscr{S})-\alpha\,\widetilde{T}(\mathscr{S},\mathscr{R}_0)
        \big). $$
    In particular, evaluating the right-hand side at the optimal allocation \(\mathscr{S}^\star\),
    $$f(\mathscr{S}_1,\mathscr{R}_0)
        \;\geq\;
        \frac{1}{2}
        \Big(
            F(\mathscr{S}^\star)-\alpha\,\widetilde{T}(\mathscr{S}^\star,\mathscr{R}_0)
        \Big).$$
    Since, by Theorem~\ref{thm:additive_optimality}, $\mathscr{R}_0$ is optimal for the modular approximation, by Propositions~\ref{prop:bounded_distance} and \ref{prop:join_bounded_distance},
    \[
        \begin{aligned}
        f(\mathscr{S}_1,\mathscr{R}_0)
        &\geq
        \frac{1}{2}
        \Big(
            F(\mathscr{S}^\star)
            -\alpha\big(
                T(\mathscr{S}^\star,\mathscr{R}^\star)
                +\kappa\delta\,|\bar{\mathcal{R}}^\star|
            \big)
        \Big) \\
        &=
        \frac{1}{2}
        \big(
            F(\mathscr{S}^\star)-\alpha\,T(\mathscr{S}^\star,\mathscr{R}^\star)
        \big)
        -\frac{\alpha}{2}\kappa\delta\,|\bar{\mathcal{R}}^\star| \\
        &=
        \frac{1}{2}\,f(\mathscr{S}^\star,\mathscr{R}^\star)
        -\frac{\alpha}{2}\kappa\delta\,|\bar{\mathcal{R}}^\star|.
        \end{aligned}
    \]
    Finally, using again the monotonicity of Algorithm~\ref{alg:hier_greedy},
    $$f(\bar{\mathscr{S}},\bar{\mathscr{R}})
        \;\geq\;
        f(\mathscr{S}_1,\mathscr{R}_0)
        \;\geq\;
        \frac{1}{2}\,f(\mathscr{S}^\star,\mathscr{R}^\star)
        -\frac{\alpha}{2}\kappa\delta\,|\bar{\mathcal{R}}^\star|,$$
    concluding the proof.
\end{proof}

\begin{rem}[Non-vacuous regime]
\label{cor:non_vacuous}
The bound in Theorem~\ref{thm:optimality} is non-vacuous whenever the profit is greater than $\alpha \kappa \delta |\bar{\mathcal{R}}^\star|$ where $\alpha\in[0,1]$, $\kappa$ and $\delta$ are generally small, see Section~\ref{sec:case}, and $|\bar{\mathcal{R}}^\star|\leq\text{diam}(\mathcal{G})$.
\end{rem}

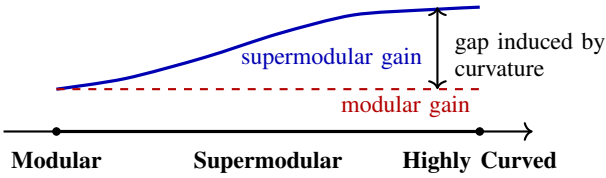
\begin{figure}[t]
\centering
\begin{tikzpicture}[scale=0.7, every node/.style={font=\small}]

    \draw[thick,->] (0,0) -- (10,0) node[right] {};

    \filldraw[black] (1,0) circle (2pt);
    \filldraw[black] (9,0) circle (2pt);

    \node[below=4pt] at (1,0) {\textbf{Modular}};

    \node[below=4pt] at (5,0) {\textbf{Supermodular}};

    \node[below=4pt] at (9,0) {\textbf{Highly Curved}};

    \draw[line width=1.2pt] (1,0) -- (9,0);
    \draw[very thick, blue!70!black]
        plot[smooth] coordinates {(1,0.8) (2.3,1.0) (3.8,1.4) (5.2,1.85) (6.6,2.2) (8.0,2.3) (9,2.35)};

    \draw[dashed, red!70!black, thick] (1,0.8) -- (9,0.8);

    \node[blue!70!black, above] at (6.2,1.0) {supermodular gain};
    \node[red!70!black, above] at (7.6,0.1) {modular gain};

    \draw[<->, thick] (8.2,0.8) -- (8.2,2.35);
    \node[right, align=left] at (8.35,1.45) {gap induced by\\curvature};

\end{tikzpicture}
\caption{Conceptual representation of distance from modularity within the class of supermodular functions. Modular functions have no perturbation $\delta$, while increasing curvature corresponds to stronger increasing returns, hence, greater perturbation magnitude.}
\label{fig:curvature_modularity}
\end{figure}

\subsection{Complexity Analysis}

We clarify the computational cost of each iteration of Algorithm~\ref{alg:hier_greedy} and the total number of iterations until convergence. Each iteration consists of two steps. On the one hand, the \emph{routing update} (line $3$) uses Dijkstra's Algorithm \cite{dijkstra1959note} to compute $\kappa$ routing paths where each shortest-path computation requires $\mathcal{O}(|E| \log |V|)$ time. Then, the \emph{allocation update} (line $5$), is done using \texttt{SGA} which requires $\mathcal{O}(|Q| \cdot \kappa)$. Hence, it yields a total per-iteration complexity of $\mathcal{O}(\kappa \, |E| \log |V| + |Q| \kappa)$.

By Lemma~\ref{lem:bounded_improvement}, \texttt{WHIRL} generates a monotonically non-decreasing sequence over a finite feasible set $(S,R)$, and therefore converges in a finite number of iterations. In the worst case, the number of iterations is upper bounded by the number of feasible allocations, which is exponential in $|Q|$. However, in practice, the coordinate-wise structure leads to rapid convergence (typically a small number of iterations), as observed in the numerical results in Table~\ref{tab:performance-comparison}.

\section{Numerical Example}
\label{sec:case}

We demonstrate the proposed framework using synthetic transportation networks of increasing scale, designed to capture both modular and nonlinear routing regimes. The experiments pursue three objectives: 
(i) illustrate the complexity increase from linear to nonlinear transportation costs, 
(ii) empirically validate the hierarchical decoupling of Section~\ref{sec:problem_structure}, and 
(iii) evaluate the performance of the proposed \texttt{WHIRL} algorithm against baseline strategies.

All simulations were executed on a laptop-grade CPU (2.6 GHz Intel i7), ensuring that scalability claims reflect realistic computational settings\footnote{Code available at  \href{https://github.com/joanvendrell/WHIRL}{\texttt{WHIRL}}.}.

\subsection{Experimental Setup}

We consider a family of directed grid-like networks $\mathcal{G} = (\mathcal{V}, \mathcal{E}, w)$,
with increasing size and connectivity, Table~\ref{tab:performance-comparison}. Nodes represent transportation hubs and edges encode baseline traversal costs. A subset of $\kappa$ nodes is randomly selected as providers, while a fixed sink node represents the client.

Each provider $i$ is endowed with a submodular utility function
$f_i(\mathcal{S}_i)$, and the transportation cost is defined as
\begin{equation}
    \label{eqn:cost_utility_function}
    T(\mathscr{S},\mathscr{R}) = \sum_{\mathcal{R}_i\in\mathscr{R}} \sum_{e\in\mathcal{R}_i} |\bigcup_{\substack{\mathcal{S}_j\in\mathscr{S}\\ e\in \mathcal{R}_j}}\mathcal{S}_j|^\beta\cdot w_e,
\end{equation}
where $\beta\in\mathbb{R}_{\geq 0}$ controls the curvature of the cost function. Note that $\beta = 0$ stands for modular cost, and supermodularity effects grows with $\beta$. We compare the proposed \texttt{WHIRL} algorithm against two mentioned strategies in Section~\ref{sec::intro}:
\begin{itemize}
    \item \textbf{Expanded-domain greedy (\texttt{ED-SGA}):} \texttt{SGA}  applied over the lifted domain $X^\prime$, approximating the joint allocation-routing problem directly \cite{conforti2010extended}.
    \item \textbf{Isolated baseline:} welfare optimization and shortest-path routing solved independently, ignoring the feedback between assignment and transportation \cite{Antil2013}.
\end{itemize}

For all experiments, the reported objective value is normalized with respect to the best solution found among the three methods on each instance.

\subsection{Synthetic Network Instances}

We generated five representative directed networks with increasing scale and routing complexity. Table~\ref{tab:performance-comparison} summarizes their structural properties. As the network grows, both the number of feasible provider-to-client routes and the coupling induced by edge sharing increase, making direct joint optimization increasingly difficult.

\subsection{Performance Comparison}

Table~\ref{tab:performance-comparison} reports the normalized objective value (Accuracy) and runtime for the three methods across the five network instances. The results show a clear trade-off between scalability and optimality.

\begin{table}[t]
\centering
\caption{Comparison of utility and runtime across methods for $\beta=2$.}
\label{tab:performance-comparison}

\begin{tabular}{c c r r | c r}
\toprule
\textbf{Net.} &
\textbf{Method} &
\textbf{Utility (\$)} $\uparrow$ &
\textbf{Time (s)} $\downarrow$ &
\multicolumn{2}{c}{\textbf{Properties}} \\
\midrule

\multirow{4}{*}{G1}
& Isolated        & 110.72 & 0.003 & $|\mathcal{V}|$  & 50  \\
& \texttt{ED-SGA} & 164.19 & 0.069 & \textit{Density} & 0.2 \\
& \texttt{WHIRL}  & 164.19 & 0.031 & $\kappa$         & 5   \\
&                 &        &       & $|Q|$            & 15  \\
\midrule

\multirow{4}{*}{G2}
& Isolated        & 222.22 & 0.283 & $|\mathcal{V}|$ & 100 \\
& \texttt{ED-SGA} & 358.35 & 13.73 & \textit{Density}& 0.1 \\
& \texttt{WHIRL}  & 362.72 & 6.587 & $\kappa$        & 8   \\
&                 &        &       & $|Q|$           & 80  \\
\midrule

\multirow{4}{*}{G3}
& Isolated        & 225.74 & 0.417 & $|\mathcal{V}|$ & 200 \\
& \texttt{ED-SGA} & 231.29 & 11.16 & \textit{Density}& 0.3 \\
& \texttt{WHIRL}  & 231.29 & 7.145 & $\kappa$        & 5   \\
&                 &        &       & $|Q|$           & 20  \\
\midrule

\multirow{4}{*}{G4}
& Isolated        & 622.96 & 0.861 & $|\mathcal{V}|$ & 300 \\
& \texttt{ED-SGA} & 638.71 & 15.28 & \textit{Density}& 0.4 \\
& \texttt{WHIRL}  & 633.21 & 8.328 & $\kappa$        & 10  \\
&                 &        &       & $|Q|$           & 55  \\
\midrule

\multirow{4}{*}{G5}
& Isolated        & 1160.4 & 1.264 & $|\mathcal{V}|$ & 500 \\
& \texttt{ED-SGA} & 1201.1 & 136.2 & \textit{Density}& 0.5 \\
& \texttt{WHIRL}  & 1185.0 & 35.11 & $\kappa$        & 15  \\
&                 &        &       & $|Q|$           & 100 \\
\bottomrule
\end{tabular}
\end{table}

The results highlight three main observations. First, \emph{\texttt{ED-SGA}} achieves near-optimal performance on small and medium instances, but its runtime grows rapidly with the size of the lifted search space, and it becomes impractical on the largest network. Second, \texttt{WHIRL} consistently attains the best or near-best objective value while remaining computationally efficient, confirming that the hierarchical decomposition preserves most of the benefit of joint optimization without incurring the combinatorial burden of expanded-domain search, see Figure~\ref{fig:utility_time_tradeoff}. Third, the \emph{isolated baseline} is computationally cheap but suffers a significant degradation in solution quality as the network grows, showing that solving assignment and routing independently fails to capture the structural coupling induced by the network.

Overall, these experiments support the central claim of the paper: \texttt{WHIRL} achieves a favorable trade-off between optimality and scalability, making it suitable for coupled welfare-routing problems where direct optimization over the full combinatorial domain is computationally prohibitive.

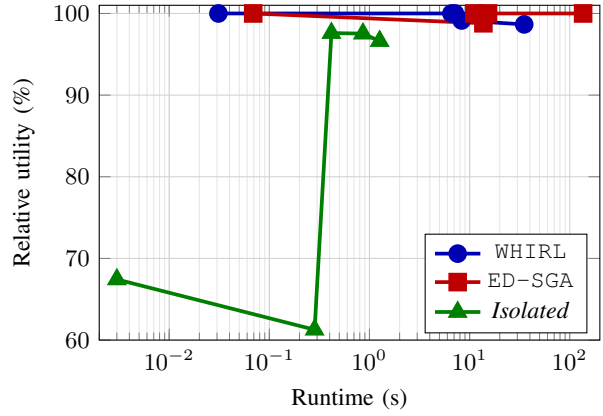
\begin{figure}[t]
\centering
\begin{tikzpicture}
\begin{axis}[
    width=0.95\linewidth,
    height=6cm,
    xlabel={Runtime (s)},
    ylabel={Relative utility (\%)},
    xmode=log,
    xmin=0.002, xmax=200,
    ymin=60, ymax=101,
    grid=both,
    minor grid style={gray!20},
    major grid style={gray!35},
    tick label style={font=\small},
    label style={font=\small},
    legend style={
        at={(0.65,0.02)},
        anchor=south west,
        font=\small,
        fill=white
    },
    every axis plot/.append style={line width=1.4pt},
]

\addplot[
    color=blue!70!black,
    mark=*,
    mark size=2.8pt,
    mark options={fill=blue!70!black},
] coordinates {
    (0.031, 100.00)
    (6.587, 100.00)
    (7.145, 100.00)
    (8.328, 99.14)
    (35.11, 98.66)
};
\addlegendentry{\texttt{WHIRL}}

\addplot[
    color=red!75!black,
    mark=square*,
    mark size=2.8pt,
    mark options={fill=red!75!black},
] coordinates {
    (0.069, 100.00)
    (13.73, 98.80)
    (11.164, 100.00)
    (15.28, 100.00)
    (136.15, 100.00)
};
\addlegendentry{\texttt{ED-SGA}}

\addplot[
    color=green!50!black,
    mark=triangle*,
    mark size=3pt,
    mark options={fill=green!50!black},
] coordinates {
    (0.003, 67.43)
    (0.283, 61.27)
    (0.417, 97.60)
    (0.861, 97.53)
    (1.264, 96.62)
};
\addlegendentry{\emph{Isolated}}

\end{axis}
\end{tikzpicture}
\caption{Relative utility--runtime trade-off across network instances. Relative utility is computed with respect to the best method per instance. \texttt{WHIRL} consistently matches or closely approaches the best utility while requiring significantly less runtime than \texttt{ED-SGA}, particularly in larger networks.}
\label{fig:utility_time_tradeoff}
\end{figure}

\subsection{Effect of Modular Deviation}

To illustrate the impact of nonlinear welfare, we tested $\beta \in \{0, 1, 2, 3, 4\}$. Greater values of $\beta$ correspond to more curved utility functions, thus amplifying diminishing returns and making assignment decisions more structurally significant. In contrast, when $\beta$ is close to $0$, the welfare term becomes nearly modular and the advantage of coupling-aware optimization is reduced.

\begin{rem}[Relation $\beta - \delta$]
    \label{rem:relation_modular_cost}
    Given \eqref{eqn:cost_utility_function}, the modular case is captured under $\beta = 0$. Hence, \eqref{eqn:perturbation} can be re-framed to
    $|\bigcup_{\substack{\mathcal{S}_j\in\mathscr{S}\\ e\in \mathcal{R}_j}}\mathcal{S}_j|^\beta\cdot w_e = w_e + \Delta_e(\mathscr{S})$, therefore 
    $\Delta_e(\mathscr{S}) = (|\bigcup_{\substack{\mathcal{S}_j\in\mathscr{S}\\ e\in \mathcal{R}_j}}\mathcal{S}_j|^\beta - 1)\cdot w_e$. Finally, by observing that $|\bigcup_{\substack{\mathcal{S}_j\in\mathscr{S}\\ e\in \mathcal{R}_j}}\mathcal{S}_j|\leq|\mathcal{Q}|$, we can conclude that $\beta$ and $\delta$ have a directly proportional relationship $$\delta \propto |\mathcal{Q}^\beta-1|\cdot\max_{e\in\mathcal{E}}w_e.$$
\end{rem}

Empirically, we observed that the performance gap between \texttt{WHIRL} and the isolated baseline increases as $\beta$ increases. This is consistent with the theory: when modular deviation is more pronounced, the objective becomes more sensitive to assignment quality, and therefore solving routing and welfare independently becomes increasingly suboptimal, see Figure~\ref{fig:beta_sensitivity}. By contrast, \texttt{WHIRL} preserves a strong objective value across all regimes, while \texttt{ED-SGA} remains competitive in solution quality but becomes more computationally expensive.

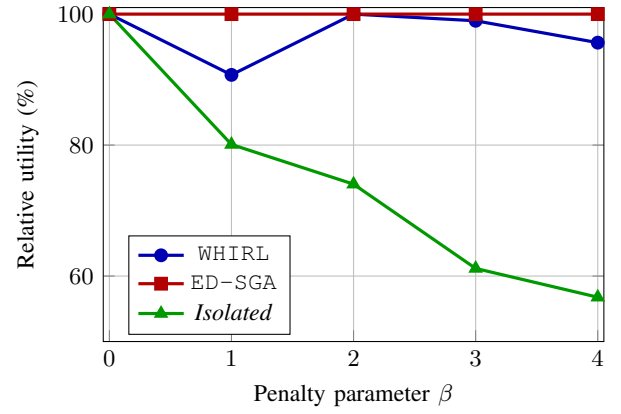
\begin{figure}[t]
\centering
\begin{tikzpicture}
\begin{axis}[
    width=0.95\linewidth,
    height=6cm,
    xlabel={Penalty parameter $\beta$},
    ylabel={Relative utility (\%)},
    xmin=-0.05, xmax=4.05,
    ymin=50, ymax=101,
    xtick={0,1,2,3,4},
    grid=both,
    legend style={at={(0.05,0.02)}, anchor=south west, font=\small},
    tick label style={font=\small},
    label style={font=\small},
]

\addplot[
    color=blue!70!black,
    mark=*,
    line width=1.2pt,
] coordinates {
    (0,100)
    (1,90.72)
    (2,100)
    (3,98.98)
    (4,95.64)
};
\addlegendentry{\texttt{WHIRL}}

\addplot[
    color=red!70!black,
    mark=square*,
    line width=1.2pt,
] coordinates {
    (0,100)
    (1,100)
    (2,100)
    (3,100)
    (4,100)
};
\addlegendentry{\texttt{ED-SGA}}

\addplot[
    color=green!60!black,
    mark=triangle*,
    line width=1.2pt,
] coordinates {
    (0,100)
    (1,80.09)
    (2,74.03)
    (3,61.16)
    (4,56.77)
};
\addlegendentry{\emph{Isolated}}

\end{axis}
\end{tikzpicture}
\caption{Average normalized objective as a function of the penalty parameter $\beta$. Greater $\beta$ amplifies diminishing returns, increasing the advantage of the proposed coupled optimization scheme.}
\label{fig:beta_sensitivity}
\end{figure}

\section{Conclusion}
\label{sec:conclusion}

This paper studied the joint problem of submodular welfare maximization and routing under both modular and supermodular transportation costs. We showed that, despite the coupling induced by shared network resources, the composite objective preserves submodularity in the allocation variable, enabling a principled hierarchical decomposition. Building on this structure, we proposed \texttt{WHIRL}, a coordinate-wise greedy algorithm with convergence guarantees and curvature-dependent approximation bounds that explicitly capture the deviation from modular routing. Numerical experiments demonstrated that the proposed approach achieves a favorable trade-off between optimality and scalability, significantly outperforming decoupled baselines while avoiding the combinatorial burden of expanded-domain optimization. Future work will focus on tightening curvature-dependent guarantees, extending the framework to stochastic and dynamic networks, and exploring learning-based routing models to further improve adaptability in large-scale systems.

\bibliographystyle{ieeetr}
\bibliography{main}

\end{document}